\documentclass[]{amsart}
\usepackage{nima}
\usepackage{xnotes}
\usepackage{pdfa}
\usepackage{amsaddr}

\usepackage{tikz}
\usetikzlibrary{decorations.markings}
\usetikzlibrary{arrows}

\title{Two-Dimensional Systolic Complexes Satisfy
  \texorpdfstring{Property~A}{Property A}}
\author{Nima Hoda}
\address{Department of Mathematics and Statistics, McGill University \\
  Burnside Hall, Room 1005 \\
  805 Sherbrooke Street West \\
  Montreal, QC, H3A 0B9, Canada}
\email{nima.hoda@mail.mcgill.ca}
\author{Damian Osajda}
\address{Instytut Matematyczny,
	Uniwersytet Wroc\l awski\\
	pl.\ Grun\-wal\-dzki 2/4,
	50--384 Wroc{\l}aw, Poland}
\address{Institute of Mathematics, Polish Academy of Sciences\\
	\'Sniadeckich 8, 00-656 War\-sza\-wa, Poland}
\email{dosaj@math.uni.wroc.pl}

\date{\today}

\theoremstyle{plain}
\newtheorem*{mainthm*}{Main~Theorem}

\newcommand{\mainthmlabel}[1]{\label{mainthm:#1}}

\DeclareMathOperator{\CAT}{CAT}
\newcommand{\Lemitmrefp}[2]{Lemma~\lemref{#1}(\itmref{#2})}

\begin{document}

\begin{abstract}
  We show that \texorpdfstring{$2$-dimensional}{2-dimensional}
  systolic complexes are quasi-isometric to quadric complexes with
  flat intervals.  We use this fact along with the weight function of
  Brodzki, Campbell, Guentner, Niblo and Wright
  \texorpdfstring{\cite{Brodzki:2009}}{} to prove that
  \texorpdfstring{$2$-dimensional}{2-dimensional} systolic complexes
  satisfy \texorpdfstring{Property~A}{Property A}.
\end{abstract}

\maketitle

\section{Introduction}

Property~A is a quasi-isometry invariant of metric spaces introduced
by Guoliang Yu in his study of the Baum-Connes conjecture
\cite{Yu2000}.  It may be thought of as a non-equivariant version of
amenability.  As is the case for amenability, Property~A has plenty of
equivalent formulations.  In particular, for finitely generated
groups, it is equivalent to the exactness of the reduced
$C^{\ast}$--algebra of the group and also to the existence of an
amenable action on a compact space \cite{HigsonRoe2000,Ozawa2000}.
Property~A implies coarse embeddability into Hilbert space and hence
the coarse Baum-Connes Conjecture and the Novikov Conjecture.  Classes
of groups for which Property~A holds include Gromov hyperbolic groups
\cite{Adams1994}, CAT(0) cubical groups \cite{CampbellNiblo2005}, and
uniform lattices in affine buildings \cite{Campbell2009}, but it is an
open question whether it holds for all CAT(0) groups.

In the current article we prove the following.
\begin{mainthm*}
  \mainthmlabel{mainthm} Two-dimensional systolic complexes satisfy
  Property~A.
\end{mainthm*}

A $2$-dimensional \defterm{systolic complex} can be defined as a
$2$-dimensional simplicial complex which is CAT(0) when equipped with
a metric in which each triangle is isometric to an equilateral
Euclidean triangle.  The class of isometry groups of such complexes is
vast.  It contains many Gromov hyperbolic groups.  It also contains
lattices in $\tilde A_2$ buildings, which were already proved to
satisfy Property~A by Campbell \cite{Campbell2009}.  Some of these
groups satisfy interesting additional properties such as Kazhdan's
property (T).  Notably, there are numerous well developed techniques
for constructing groups acting on $2$-dimensional systolic complexes
(with various additional features) making them a rich source of
examples.  For instance, given a finite group $F$ and a generating set
$S \subseteq F \setminus \{1\}$ whose Cayley graph $\Gamma(F,S)$ has
girth at least 6, Ballmann and {\' S}wi{\c a}tkowski \cite[Theorem~2
and Section~4]{Ballmann:1997} construct a canonical infinite
$2$-dimensional systolic complex $X$ whose oriented triangles are
labeled by elements of $S \cup S^{-1}$ and in which the link of every
vertex is isomorphic to $\Gamma(F,S)$ with labels induced from the
triangles.  The labeled automorphisms of $X$ act simply transitively
on the oriented $1$-simplices of $X$ and if $(F,S)$ has a relation of
the form $(st)^3$ with $s,t\in S$ then $X$ has flat planes and so is
not hyperbolic.  This construction is a particular case of a
development of a complex of groups as described in Bridson and
Haefliger \cite[Example~4.19(2) of
Chapter~III.$\mathcal{C}$]{bridson1999metric}.

We prove the Main~Theorem by showing that $2$-dimensional systolic
complexes are quasi-isometric to quadric complexes whose intervals
with respect to a basepoint are $\CAT(0)$ square complexes
(\Thmref{sqqi}, \Thmref{sqquad} and \Thmref{sqfi}).  This allows us to
apply the weight function and uniform convergence argument of Brodzki,
Campbell, Guentner, Niblo and Wright \cite{Brodzki:2009} in their
proof that finitely dimensional $\CAT(0)$ cube complexes satisfy
Property~A (\Thmref{qfia}).

\subsection*{Acknowledgements}
The authors would like to thank Jacek {\' S}wi{\c a}tkowski for some
helpful discussions on the construction described above of
$2$-dimensional systolic developments.  N.H.\ was partially funded by
an NSERC CGS M.  D.O.\ was partially supported by (Polish) Narodowe
Centrum Nauki, grant no.\ UMO-2015/18/M/ST1/00050.  Parts of this
research were carried out while D.O.\ was visiting McGill University
and while N.H.\ was visiting the University of Wroc\l aw.  The authors
would like to thank both institutions for their hospitality.

\section{Preliminaries}

For basic algebraic topological notions such as those of \defterm{CW
  complexes} and \defterm{simple connectedness} we refer the reader to
Hatcher \cite{Hatcher:2002}.  A \defterm{combinatorial map} $X \to Y$
between CW complexes is one whose restriction to each cell of $X$ is a
homeomorphism onto a cell of $Y$.  All graphs considered in this paper
are simplicial.  We consider metrics only on the $0$-skeleton $X^0$ of
a cell complex $X$ as induced by the shortest path metric on the
$1$-skeleton $X^1$.  The \defterm{metric ball} $B_r(v)$ (respectively,
\defterm{metric sphere} $S_r(v)$) of radius $r$ centered at a vertex
$v$ in a complex is the subgraph induced in the $1$-skeleton by the
set of vertices of distance at most (respectively, exactly) $r$ from
$v$.  The \defterm{girth} of a graph is the length of the embedded
shortest cycle.  The \defterm{link} of a vertex $v$ in a
$2$-dimensional simplicial complex is the graph whose vertices are the
neighbours of $v$ and where two vertices are joined by an edge if they
span a triangle together with $v$.

\subsection{Property A}

Rather than defining Property~A in full generality, we give the
characterization for graphs of Brodzki, Campbell, Guentner, Niblo and
Wright \cite[Proposition~1.5]{Brodzki:2009}.  A graph $\Gamma$
satisfies \defterm{Property~A} iff there exists a sequence of
constants $(C_n)_{n\in\N}$ and a family of functions
$f_{n,v}\colon \Gamma^0 \to \N$ indexed by $\N \times \Gamma^0$ such
that the following conditions hold.

\begin{enumerate}
\item $f_{n,v}$ is supported on $B_{C_n}(v)$.
\item $\frac{||f_{n,v} - f_{n,v'}||_1}{||f_{n,v}||_1} \to 0$ uniformly
  over all pairs of vertices $(v,v')$ joined by an edge.
\end{enumerate}

\subsection{Two-dimensional systolic complexes}
A \defterm{$2$-dimensional systolic complex} is a simply connected
$2$-dimensional simplicial complex in which the girth of the link of
every vertex is at least 6.

The following are well known properties of systolic complexes.  See
Chepoi \cite[Theorem~8.1]{Chepoi:2000} and Januszkiewicz and
{\'S}wi{\c{a}}tkowski \cite[Lemma~7.7]{Januszkiewicz:2006}.

\begin{lem}
  \lemlabel{systlem} Let $Y$ be a $2$-dimensional systolic complex.
  \begin{enumerate}
  \item \itmlabel{sphtfree} Metric spheres in $Y$ are triangle-free.
  \item \itmlabel{ballneighb} Let $u,v \in Y^0$.  Let $w,x \in Y^0$ be
    neighbours of $v$ that are closer to $u$ than is $v$.  Then $w$
    and $x$ are joined by an edge.
  \item \itmlabel{tricond} (Triangle Condition) Let $u,v,w \in Y^0$
    such that $v$ and $w$ are joined by an edge and they are
    equidistant to $u$.  Then there exists $x \in Y^0$ adjacent to $v$
    and $w$ and closer to $u$ than are $v$ and $w$.
  \end{enumerate}
\end{lem}

\subsection{Quadric complexes}

\begin{figure}
  \begin{subfigure}{\textwidth}
    \centering
    \begin{tikzpicture}
      \tikzstyle{vertex} = [circle,minimum size=0.15cm,inner sep=0,fill=black];
      \begin{scope}
        \node[vertex,label={above:$u_0$}] (u0) at (1, 2) {};
        \node[vertex,label={below:$u_1$}] (u1) at (1, 0) {};
        \node[vertex,label={left:$v_0$}] (v0) at (0, 1) {};
        \node[vertex,label={left:$v_1$}] (v1) at (1, 1) {};
        \node[vertex,label={right:$v_2$}] (v2) at (2, 1) {};

        \foreach \u in {u0,u1}
          \foreach \v in {v0,v1,v2}
            \draw[thick] (\u) -- (\v);
      \end{scope}

      \draw[-implies, double equal sign distance] (3,1) -- (4,1);
      
      \begin{scope}[xshift=5cm]
        \node[vertex,label={above:$u_0$}] (u0) at (1, 2) {};
        \node[vertex,label={below:$u_1$}] (u1) at (1, 0) {};
        \node[vertex,label={left:$v_0$}] (v0) at (0, 1) {};
        \node[vertex,label={right:$v_2$}] (v2) at (2, 1) {};

        \foreach \u in {u0,u1}
        \foreach \v in {v0,v2}
        \draw[thick] (\u) -- (\v);
      \end{scope}
    \end{tikzpicture}
    \caption{}
    \figlabel{repl2}
  \end{subfigure}
  \\
  \begin{subfigure}{\textwidth}
    \centering
    \begin{tikzpicture}
      \tikzstyle{vertex} = [circle,minimum size=0.15cm,inner sep=0,fill=black];

      \begin{scope}
        \node[vertex,label={above right:$u$}] (u) at (0, 0) {};
        \node[vertex,label={above:$b_0$}] (b0) at (30:1) {};
        \node[vertex,label={above:$a_0$}] (a0) at (90:1) {};
        \node[vertex,label={above:$b_2$}] (b2) at (150:1) {};
        \node[vertex,label={below:$a_2$}] (a2) at (210:1) {};
        \node[vertex,label={below:$b_1$}] (b1) at (270:1) {};
        \node[vertex,label={below:$a_1$}] (a1) at (330:1) {};

        \foreach \v in {a0,a1,a2}
          \draw[thick] (\v) -- (u);

        \draw[thick] (a0) -- (b0);
        \draw[thick] (b0) -- (a1);
        \draw[thick] (a1) -- (b1);
        \draw[thick] (b1) -- (a2);
        \draw[thick] (a2) -- (b2);
        \draw[thick] (b2) -- (a0);
      \end{scope}

      \draw[-implies,double equal sign distance] (1.5,0) -- (2.5,0);
      
      \begin{scope}[xshift=4cm]
        \node[vertex,label={above:$b_0$}] (b0) at (30:1) {};
        \node[vertex,label={above:$a_0$}] (a0) at (90:1) {};
        \node[vertex,label={above:$b_2$}] (b2) at (150:1) {};
        \node[vertex,label={below:$a_2$}] (a2) at (210:1) {};
        \node[vertex,label={below:$b_1$}] (b1) at (270:1) {};
        \node[vertex,label={below:$a_1$}] (a1) at (330:1) {};

        \draw[thick] (a0) -- (b0);
        \draw[thick] (b0) -- (a1);
        \draw[thick] (a1) -- (b1);
        \draw[thick] (b1) -- (a2);
        \draw[thick] (a2) -- (b2);
        \draw[thick] (b2) -- (a0);

        \draw[thick] (b2) -- (a1);
      \end{scope}

      \node at (5.5,0) {,};

      \begin{scope}[xshift=7cm]
        \node[vertex,label={above:$b_0$}] (b0) at (30:1) {};
        \node[vertex,label={above:$a_0$}] (a0) at (90:1) {};
        \node[vertex,label={above:$b_2$}] (b2) at (150:1) {};
        \node[vertex,label={below:$a_2$}] (a2) at (210:1) {};
        \node[vertex,label={below:$b_1$}] (b1) at (270:1) {};
        \node[vertex,label={below:$a_1$}] (a1) at (330:1) {};

        \draw[thick] (a0) -- (b0);
        \draw[thick] (b0) -- (a1);
        \draw[thick] (a1) -- (b1);
        \draw[thick] (b1) -- (a2);
        \draw[thick] (a2) -- (b2);
        \draw[thick] (b2) -- (a0);

        \draw[thick] (a0) -- (b1);
      \end{scope}

      \node at (8.5,0) {,};

      \begin{scope}[xshift=10cm]
        \node[vertex,label={above:$b_0$}] (b0) at (30:1) {};
        \node[vertex,label={above:$a_0$}] (a0) at (90:1) {};
        \node[vertex,label={above:$b_2$}] (b2) at (150:1) {};
        \node[vertex,label={below:$a_2$}] (a2) at (210:1) {};
        \node[vertex,label={below:$b_1$}] (b1) at (270:1) {};
        \node[vertex,label={below:$a_1$}] (a1) at (330:1) {};

        \draw[thick] (a0) -- (b0);
        \draw[thick] (b0) -- (a1);
        \draw[thick] (a1) -- (b1);
        \draw[thick] (b1) -- (a2);
        \draw[thick] (a2) -- (b2);
        \draw[thick] (b2) -- (a0);

        \draw[thick] (a2) -- (b0);
      \end{scope}
    \end{tikzpicture}
    \caption{}
    \figlabel{repl3}
  \end{subfigure}
  \caption{Replacement rules for quadric complexes.}
  \figlabel{repl}
\end{figure}
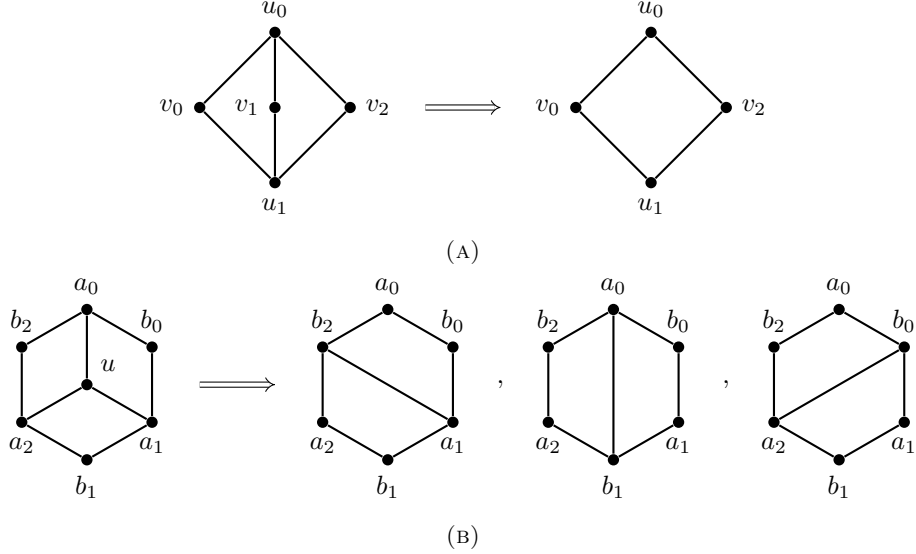

A \defterm{square complex} is a $2$-dimensional CW complex with a
simplicial $1$-skeleton such that the attaching maps of $2$-cells are
injective combinatorial maps from $4$-cycles.  The closed $2$-cells of
a square complex are called \defterm{squares}.  We assume that no two
squares of a square complex are glued to the same $4$-cycle of the
$1$-skeleton.  A \defterm{quadric complex} is a simply connected
square complex in which, for any subcomplex as on the left-hand side
of \Figref{repl}, there exists at least one subcomplex as on its
right-hand side having the same boundary path.  In other words,
quadric complexes are simply connected generalized $(4,4)$-complexes,
as defined by Wise \cite{Wise:2003}, that are built out of squares.
They were first studied in depth by Hoda \cite{Hoda:2017}.  A quadric
complex is a $\CAT(0)$ square complex if and only if its $1$-skeleton
is $K_{2,3}$-free.

\begin{thm}[Hoda \cite{Hoda:2017}]
  \thmlabel{ballisom} Let $X$ be a quadric complex.  Metric balls are
  isometrically embedded in $X^1$.
\end{thm}

\begin{lem}[Quadrangle Condition]
  \lemlabel{quadcond} Let $X$ be a quadric complex.  Let
  $u,v,w,x \in Y^0$ such that $v$ and $w$ are adjacent to $x$ and $v$
  and $w$ are closer to $u$ than is $x$.  Then there exists
  $y \in X^0$ adjacent to $v$ and $w$ and closer to $u$ than are $v$
  and $w$.
\end{lem}

\Lemref{quadcond} follows from the fact that the $1$-skeleta of
quadric complexes are precisely the hereditary modular graphs.  This
characterization is due to the result of Hoda that the $1$-skeleta are
the bi-bridged graphs \cite{Hoda:2017} and the theorem of Bandelt that
a graph is bi-bridged iff it is hereditary modular
\cite[Theorem~1]{Bandelt:1988}.

Let $X$ be a quadric complex.  The \defterm{interval} $I(u,v)$ between
a pair of vertices $u$ and $v$ in $X^0$ is the full subcomplex induced
by the union of the geodesics between $u$ and $v$.

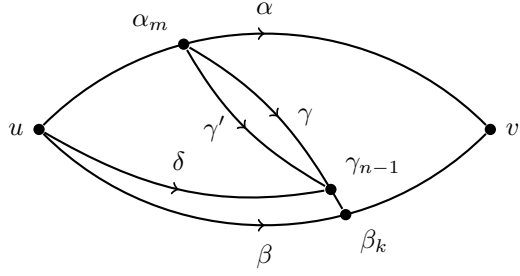
\begin{figure}
  \centering
  \begin{tikzpicture}
    \tikzstyle{vertex} = [circle,minimum size=0.15cm,inner sep=0,fill=black];

    \node[vertex,label={left:$u$}] (u) at (0,0) {};
    \node[vertex,label={right:$v$}] (v) at (6,0) {};

    \draw[thick,postaction={decorate},decoration={markings,mark=at
      position 1/2 with {\arrow{>},\node[label={above:$\alpha$}]
        {};},mark=at position 1/3 with {\node[vertex,label={above
          left:$\alpha_m$}] (am) {};}}] (u) to [out=45,in=135] (v);

    \draw[thick,postaction={decorate},decoration={markings,mark=at
      position 1/2 with {\arrow{>},\node[label={below:$\beta$}]
        {};},mark=at position 2/3 with {\node[vertex,label={below
          right:$\beta_k$}] (bk) {};}}] (u) to [out=315,in=225] (v);

    \draw[thick,postaction={decorate},decoration={markings,mark=at
      position 1/2 with {\arrow{>},\node[label={right:$\gamma$}]
        {};},mark=at position 9/10 with {\node[vertex,label={above
          right:$\gamma_{n-1}$}] (gnm1) {};}}] (am) to
    [out=330,in=120] (bk);

    \draw[thick,postaction={decorate},decoration={markings,mark=at
      position 1/2 with {\arrow{>},\node[label={above:$\delta$}]
        {};}}] (u) to [out=330,in=190] (gnm1);

    \draw[thick,postaction={decorate},decoration={markings,mark=at
      position 1/2 with {\arrow{>},\node[label={left:$\gamma'$}]
        {};}}] (am) to [out=300,in=150] (gnm1);
  \end{tikzpicture}
  \caption{Two geodesics $\alpha$ and $\beta$ between a pair of
    vertices $u$ and $v$ of a quadric complex along with a geodesic
    $\gamma$ joining $\alpha_m \in \alpha$ and $\beta_k \in \beta$.
    Used in the proof of \Thmref{intisomet}.}
  \figlabel{intisomet}
\end{figure}

\begin{thm}
  \thmlabel{intisomet} Let $X$ be a quadric complex and let
  $u,v \in X^0$.  The $1$-skeleton of $I(u,v)$ is isometrically
  embedded in $X^1$.
\end{thm}
\begin{proof}
  Suppose not.  Then there are geodesics $(\alpha_i)_{i=0}^{\ell}$ and
  $(\beta_i)_{i=0}^{\ell}$ from $u$ to $v$ and indices $m$ and $k$
  such that no geodesic $(\gamma_i)_{i=0}^n$ from $\alpha_m$ to
  $\beta_k$ is contained in $I(u,v)$.  Choose $(\alpha_i)_i$,
  $(\beta_i)_i$, $m$ and $k$ so as to minimize
  $n = d(\alpha_m,\beta_k)$.  Without loss of generality, $m \le k$.
  
  By \Thmref{ballisom} we may assume that $(\gamma_i)_i$ is contained
  in $B_k(u)$.  Hence, since $X^1$ is bipartite,
  $d(u,\gamma_{n-1}) = k - 1$.  Let $(\delta_i)_{i=0}^{k-1}$ be a
  geodesic from $u$ to $\gamma_{n-1}$.  Let $(\beta'_i)_{i=0}^{\ell}$
  be the concatenation of the sequences $(\delta_i)_{i=0}^{k-1}$ and
  $(\beta_i)_{i=k}^{\ell}$.  Then $(\beta'_i)_i$ is a geodesic from
  $u$ to $v$.  By the minimality of
  $\bigl((\alpha_i)_i, (\beta_i)_i, m, k\bigr)$, there is a geodesic
  $(\gamma'_i)_{i=0}^{n-1}$ from $\alpha_m$ to
  $\beta'_{k-1} = \gamma_{n-1}$ such that $(\gamma'_i)_i$ is contained
  in $I(u,v)$.  But then appending $\gamma_n$ to $(\gamma'_i)_i$ we
  obtain a geodesic from $\alpha_m$ to $\beta_k$ that is contained in
  $I(u,v)$.  This is a contradiction.
\end{proof}

\begin{cor}
  \corlabel{intquad} Intervals in quadric complexes are quadric.
\end{cor}
\begin{proof}
  Quadric complexes are characterized by metric properties of their
  $1$-skeleta \cite{Hoda:2017}, so an isometrically embedded full
  subcomplex of a quadric complex is quadric.
\end{proof}

Fix a basepoint $\ast \in X^0$.  If, for all $v \in X^0$, the interval
$I(\ast, v)$ is a $\CAT(0)$ square complex then $(X,\ast)$
\defterm{has flat intervals}.  By \Corref{intquad}, $(X,\ast)$ has
flat intervals if and only if every $I(\ast,v)$ is $K_{2,3}$-free.

We now describe how we apply the results of Brodzki et al.\
\cite{Brodzki:2009} in the special case of $2$-dimensional $\CAT(0)$
cube complexes to our present situation.

Let $(X,\ast)$ be a based quadric complex with flat intervals.  Let
$Z_v = I(\ast,v)$ be a $\CAT(0)$ square complex interval in a quadric
complex.  We will describe the weight function $f_{n,v}$ of Brodzki et
al.\ \cite{Brodzki:2009} for $v$ in $Z_v$.  For $w \in Z_v^0$, let
$\rho(w)$ be the number of neighbours of $w$ in $Z_v^1$ that lie on
geodesics from $w$ to $\ast$.  The \defterm{deficiency} of
$w \in Z_v^0$ is defined as follows.
\[\delta(w) = 2 - \rho(w) \]
Define $f_{n,v}\colon Z_v^0 \to \N$ as follows.\[ f_{n,v}(w) =
\begin{cases}
  0 & \text{if $d(w,v) > n$} \\
  1 & \text{if $d(w,v) \le n$ and $\delta(w) = 0$} \\
  n - d(w,v) + 1 & \text{if $d(w,v) \le n$ and $\delta(w) = 1$} \\
  \frac{1}{2} \bigl(n - d(w,v) + 2\bigr)\bigl(n - d(w,v) + 1\bigr) &
  \text{if $d(w,v) \le n$ and $\delta(w) = 2$}
\end{cases}
\]
We extend $f_{n,v}$ by zeroes to all of $X^0$.  Note that if $v'$ is a
neighbour of $v$, then $Z_{v'} \subseteq Z_v$ or
$Z_v \subseteq Z_{v'}$, say the latter, and that $Z_v$ and $Z_{v'}$
are both intervals of $\ast$ in $Z_{v'}$, which by flatness is a
$\CAT(0)$ square complex.  So we may apply the results of Brodzki et
al.\ that
\[ ||f_{n,v}||_1 = \frac{1}{2}(n+2)(n+1) \]
\cite[Proposition~3.10]{Brodzki:2009} and for a neighbour $v'$ of
$v$, \[ ||f_{n,v} - f_{n,v'}||_1 = 2(n+1) \]
\cite[Proposition~3.11]{Brodzki:2009} and so we have the following.

\begin{thm}
  \thmlabel{qfia} Let $X$ be a quadric complex.  If there exists
  $\ast \in X^0$ such that $(X,\ast)$ has flat intervals then $X$
  satisfies Property~A.
\end{thm}

\section{The squaring of
  \texorpdfstring{$2$-dimensional}{2-dimensional} systolic complexes}

Let $(Y,\ast)$ be a $2$-dimensional systolic complex with a basepoint
$\ast \in Y^0$.  Let $X_Y^1$ be the subgraph of $Y^1$ given by the
union of all edges whose endpoints are not equidistant to $u$.  Note
that $X_Y^1$ is bipartite The \defterm{squaring} of $(Y,\ast)$ is the
based square complex $(X_Y,\ast)$ obtained from $X_Y^1$ by attaching a
unique square along its boundary to each embedded $4$-cycle of
$X_Y^1$.

\begin{thm}
  \thmlabel{sqqi} Let $(Y,\ast)$ be a based $2$-dimensional systolic
  complex and let $(X_Y,\ast)$ be the squaring of $(Y,\ast)$.  Then
  $Y^1$ is quasi-isometric to $X_Y^1$.
\end{thm}
\begin{proof}
  Applying \Lemitmrefp{systlem}{tricond} to $\ast$ and an edge $e$ of
  $Y^1$ that is not contained in $X_Y^1$ gives us a triangle, one of
  whose edges is $e$ and whose remaining edges are contained in
  $X_Y^1$.  This ensures that distances in $X_Y^1$ increase by at most
  a factor of two relative to distances in $Y^1$.
\end{proof}

\begin{thm}
  \thmlabel{sqquad} Let $(Y,\ast)$ be a based $2$-dimensional systolic
  complex.  The squaring $(X_Y,\ast)$ of $(Y,\ast)$ is quadric.
\end{thm}
\begin{proof}
  We need to show that $X_Y$ is simply connected and that, for every
  subgraph of $X_Y^1$ as in the left-hand side of \Figref{repl3}, a
  pair of antipodal vertices in the outer $6$-cycle is joined by an
  edge.
  
  To show that $X_Y$ is simply connected it suffices to show that, for
  every embedded cycle $\alpha$ of $X_Y^1$, there is a $2$-dimensional
  square complex $D'$ homeomorphic to a $2$-dimensional disk and a
  combinatorial map $D' \to X_Y$ whose restriction to the boundary
  $\bd D'$ of $D'$ is $\alpha$.  By the van Kampen Lemma
  \cite[Proposition~9.2 of Section~III.9]{Lyndon:2001} since $Y$ is
  simply connected, there exists a $2$-dimensional simplicial complex
  $D$ homeomorphic to a $2$-disk $D$ and a combinatorial map $D \to Y$
  which restricts to $\alpha$ on the boundary.  Choose such $D \to Y$
  so as to minimize the number of triangles of $D$.  By
  \Lemitmrefp{systlem}{sphtfree}, each triangle of $D$ has a unique
  edge $e$ that is not contained in $X_Y^1$.  Then $e$ is contained in
  the interior of $D$ and, by the minimality of $D \to Y$, the star of
  $e$ is embedded in $Y$.  Let $D'$ be the result of deleting all such
  $e$ from $D^1$ and then spanning a square on each embedded
  $4$-cycle.  Since every embedded $4$-cycle of $X_Y$ spans a square,
  we may extend $(D')^1 \to (X_Y)^1$ to $D' \to X_Y$.  This proves
  that $X_Y$ is simply connected.

  Let $W$ be a subgraph of $X_Y^1$ as in the left-hand side of
  \Figref{repl3} and with the same vertex labels.  By $6$-largeness of
  $Y$, each of the embedded $4$-cycles of $W$ have a diagonal.  Since
  the girth of the link of $u$ is at least 6, these diagonals must
  join $u$ to each of the $b_i$.  Hence $u$ is adjacent to every
  vertex in the outer $6$-cycle $C$ of $W$.

  Let $v$ be a furthest vertex of $C$ from $\ast$. By
  \Lemitmrefp{systlem}{ballneighb}, the neighbours of $v$ in $C$ are
  joined by an edge.  But then there is a $5$-cycle in the link of $u$
  which contradicts the $2$-dimensional systolicity of $Y$.
\end{proof}

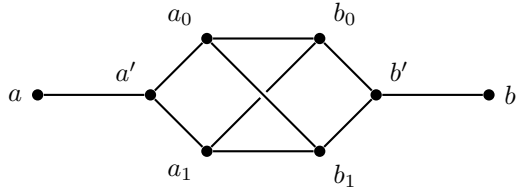
\begin{figure}
  \centering
  \begin{tikzpicture}[scale=1.5]
    \tikzstyle{vertex} = [circle,minimum size=0.15cm,inner sep=0,fill=black];
    \begin{scope}
      \node[vertex,label={above right:$b'$}] (bp) at (1,0) {};
      \node[vertex,label={above right:$b_0$}] (b0) at (1/2,1/2) {};
      \node[vertex,label={above left:$a_0$}] (a0) at (-1/2,1/2) {};
      \node[vertex,label={above left:$a'$}] (ap) at (-1,0) {};
      \node[vertex,label={below left:$a_1$}] (a1) at (-1/2,-1/2) {};
      \node[vertex,label={below right:$b_1$}] (b1) at (1/2,-1/2) {};

      \node[vertex,label={left:$a$}] (a) at (-2,0) {};
      \node[vertex,label={right:$b$}] (b) at (2,0) {};

      \draw[thick] (bp) -- (b0);
      \draw[thick] (b0) -- (a0);
      \draw[thick] (a0) -- (ap);
      \draw[thick] (ap) -- (a1);
      \draw[thick] (a1) -- (b1);
      \draw[thick] (b1) -- (bp);

      \draw[thick] (a1) -- (b0);
      \draw[thick,preaction={draw,line width=3.2pt,white}] (a0) -- (b1);
      
      \draw[thick] (bp) -- (b);
      \draw[thick] (ap) -- (a);
    \end{scope}
  \end{tikzpicture}
  \caption{A $K_{2,2}$ spanning the vertices $a_0$, $a_1$, $b_0$ and
    $b_1$ and embedded in a particular way in the interval $I(a,b)$ of
    a graph.  Such an embedding is not possible in the squaring of a
    based $2$-dimensional systolic complex.  Used in the proof of
    \Thmref{sqfi}.}
  \figlabel{tsqint}
\end{figure}

\begin{thm}
  \thmlabel{sqfi} Let $(Y,\ast)$ be a based $2$-dimensional systolic
  complex.  The squaring $(X_Y,\ast)$ of $(Y,\ast)$ has flat
  intervals.
\end{thm}
\begin{proof}
  Suppose there is a $K_{2,3}$ in an interval $I(\ast,v)$ of $X_Y$.
  Let $\{a_0,a_1\} \sqcup \{b_0, b_1, b_2\}$ be the bipartition of the
  $K_{2,3}$.  Some pair of vertices of $\{b_0, b_1, b_2\}$ are
  equidistant to $\ast$, say $\{b_0,b_1\}$.
  
  Consider the case where $a_0$ and $a_1$ are equidistant to $\ast$.
  Let $a$ be the closer of $\ast $ and $v$ to $a_0$ and let $b$ be the
  closer of $\ast$ and $v$ to $b_0$.  Let $a'$ and $b'$ be obtained by
  applying \Lemref{quadcond} to $a_0$, $a_1$ and $a$ and to $b_0$,
  $b_1$ and $b$, as in \Figref{tsqint}.  By $6$-largeness of $Y$, the
  $4$-cycle $(a_0, a', a_1,b_0)$ has a diagonal in $Y^1$.  Since
  $(a',a_0,b_0)$ is a geodesic, the diagonal must join $a_0$ and
  $a_1$.  Similarly, $b_0$ and $b_1$ are joined by edge in $Y^1$ and
  hence, by flagness, $\{a_0, a_1, b_0, b_1\}$ spans a $3$-simplex in
  $Y$.  This contradicts the $2$-dimensionality of $Y$.

  In the remaining case $a_0$ and $a_1$ are not equidistant to $\ast$.
  Then the $b_i$ must all be equidistant to $\ast$ with the
  $(a_0,b_i,a_1)$ all geodesics.  Applying a similar argument as in
  the previous case to the $4$-cycles $(a_0,b_i,a_1,b_j)$ we see that
  the $b_i$ span a triangle in $Y$ and so, together with $a_0$, they
  span a $3$-simplex contradicting, again, the $2$-dimensionality of
  $Y$.
\end{proof}

As an immediate consequence of \Thmref{sqqi}, \Thmref{sqquad},
\Thmref{sqfi} and \Thmref{qfia} we have the Main~Theorem.

\bibliographystyle{abbrv}
\bibliography{\jobname}
\end{document}